\documentclass[12pt, twoside]{article}

\usepackage{robert}
\title{\textbf{Notes on regularity of Anosov splitting}\vspace{-0.5em}}
\date{}

\author{Robert Koirala\vspace{-0.5em}}

\begin{document}

\maketitle
\vspace{-4em}
\begin{abstract}
    In these expository notes, we give a proof of regularity of Anosov splitting for Anosov diffeomorphisms in a torus. We also generalize the idea to higher dimensions and to Anosov flows.
\end{abstract}

\section{Introduction}
A dynamical system refers to an iteration of a map from a space to itself. The system is hyperbolic if any two orbits of the map diverge exponentially either in the past and/or the future. For instance, consider the action of $A=\begin{pmatrix}2&0\\0&2^{-1}\end{pmatrix}$ on a plane. In a hyperbolic system, we can characterize the set of points whose orbits remain close to the orbit of a point $x$ in the past (future) as an immersed submanifold of a Euclidean space. The submanifold is called an unstable (stable) manifold of $x$. In the case of $A,$ the stable manifold at the origin corresponds to the $y$ axis and the unstable to the $x$ axis. It turns out that the stable (unstable) manifolds are as regular as the map \cite{D}. We refer the reader to \cite{D} for an expository proof of the stable/unstable manifold theorem and \cite{KH} for a detailed account of hyperbolic dynamics.

In fact, for a volume preserving hyperbolic map, the unstable (stable) manifolds form an unstable (stable) foliation which is almost $C^2$ regular (\cite{H89}). To prove regularity, it often requires some bounds on the rates of divergence of the orbits. On the other hand, high regularity in these settings implies rigidity of the map \cite{H89}. However, regularity of the foliation for maps without bounds on the rates of divergence is still unanswered.

In these notes, we prove the almost $C^{2}$ regularity of the unstable (stable) foliation in two dimensions following the general case in \cite{H89} and \cite{HK}. In particular, we will prove the following theorem:

\begin{thm}\label{mainthm}
For any $\alpha\in (0,2]$ and $\epsilon>0$, if $\varphi$ is an $\alpha$-bunched Anosov diffeomorphism of a torus $\mathbb{T}^2$ then the unstable (stable) foliation associated to the diffeomorphism is $C^{\alpha-\epsilon}$ regular.
\end{thm}

Unless otherwise stated, diffeomorphisms in these notes mean Anosov diffeomorphisms. 

To get a feeling for the theorem, we refer the reader to Figure \ref{fig:numerical}. Note that the diagonal entries of Arnold cat map $\begin{pmatrix}2&1\\1&1\end{pmatrix}$ satisfy $0<\lambda_1<1<\lambda_2$ which implies hyperbolicity. In fact, if we perturb the map, the hyperbolicity is preserved. Therefore, we get unstable and stable manifolds at each point for the perturbed map. And Theorem \ref{mainthm} means that the blue (red) manifolds in Figure \ref{fig:numerical} vary regularly with respect to $x\in \mathbb{T}^2.$
\begin{figure}[h]
    \centering
    \includegraphics[width=0.45\textwidth]{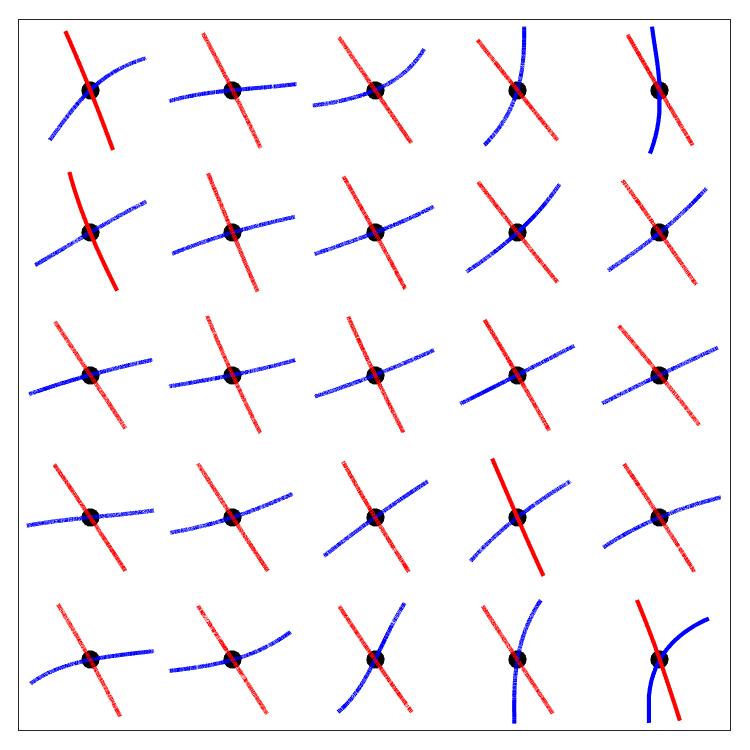}
    \caption{Numerically computed unstable (blue) and stable (red) manifolds for a perturbed Arnold cat map for different $x\in \mathbb{T}^2=\R^2/\Z^2$ \cite{D21}.}
    \label{fig:numerical}
\end{figure}

The outline of the notes is as follows:
\begin{itemize}
    \item In \S\ref{sec2}, we will focus on regularity of the unstable foliation $W_u$ of a torus. We can reverse the time to get a proof for the stable foliation. Since regularity is a local property, we will focus on a neighborhood of a point. Further, understanding how the tangent spaces $T_xW_u(x)$ of an unstable manifold at $x$ vary with $x$ suffices to understand regularity of the foliation. The gist of the proof in a general setting is same as that at a point in a torus. 
    
    The three main ingredients in the proof of the regularity theorem are:
    \begin{itemize}
        \item\textbf{ H\"older continuity} (\S\ref{holderc}): If we think of a torus $\mathbb{T}^2$ as a quotient of $\R^2$, the tangent spaces $T_xW_u(x)$ for all $x\in \mathbb{T}^2$ form an unstable vector field. We define a space $V(\delta)$ of vector fields that are $\delta$-close to the unstable vector field. Under the action of a diffeomorphism in positive time, any vector field in $V(\delta)$ converges to the unstable vector field. Now the idea is to prove that the action of the map preserves H\"older continuity for all vector fields in $V(\delta)$. And using a limiting argument, we can prove that the limiting distribution is also regular.
        
        \item \textbf{Differentiability} (\S\ref{diff}): To prove the differentiability, one might think that we have to fix a connection and define differentiation. However, we can get around with it by using the definition of differentiation as a limit of difference quotient. To get the limiting parameter, the idea is to exploit the fact that a diffeomorphism contracts stable manifolds.
        
        \item\textbf{ H\"older continuity of the derivative} (\S\ref{holdercd}): Again, the idea is to prove that the action of a diffeomorphism preserves H\"older continuity of the derivatives.
    \end{itemize}
    \item In \S\ref{general}, we will outline the proof in higher dimensions. Finally, we will comment on how the proof generalizes to Anosov flows.
\end{itemize}

We write $\N$ to be the set of non-negative integers, $\Z_+$ the set of positive integers and $\floor{x}$ the integer part of $x.$ For two functions $f$ and $g$ by $f(x)=\mathcal{O}(g(x))$ we mean $\verts{f(x)}\leq cg(x)$ where $c$ is a non-negative constant. Further, by $f(x)=o(g(x))$ we mean $\frac{f(x)}{g(x)}\to 0$ when $x$ approaches to $0.$

Unless otherwise stated, a manifold is a torus $\mathbb{T}^2$. In particular, consider an $\ell_\infty$ unit ball centered at the origin:
\begin{align}
    \bar{B}_{\infty}(0,1)=\{(x_1,x_2)\in \R^2:\max(\verts{x_1},\verts{x_2})\leq 1\}.
\end{align}
Then $\mathbb{T}^2$ is obtained by identifying the edges of $\bar{B}_{\infty}(0,1)$: $(-1,x_2)$ with $(1,x_2)$ and $(x_1,1)$ with $(x_1,-1).$ We endow $\mathbb{T}^2$ with the metric $\norm{\cdot}$ induced from $\R^2.$ Note that the choice of a metric is irrelevant because all metrics on a compact manifold are commensurate.


\section{Regularity in two dimensions}\label{sec2}
In this section, we prove regularity of Anosov splitting in two dimensions with some simplification. First, let us start with the definition of Anosov diffeomorphism with a `bunching condition.' Since we don't explicitly use the Definitions \ref{anosovdiff} and \ref{bunching}, the reader can skip them and come back when needed.

\begin{defn}[\cite{D}, \cite{H89}]\label{anosovdiff}
A volume preserving diffeomorphism $\varphi$ on a torus $\mathbb{T}^2$ is called \textit{Anosov} with \textit{Anosov splitting} $(E_u,E_s)$ if the tangent bundle $T\mathbb{T}^2$ splits into a direct sum $E_u\oplus E_s$ where $E_u$ and $E_s$ are one dimensional subbundles of $T\mathbb{T}^2$ and:
\begin{itemize}[noitemsep]
    \item $E_u$ and $E_s$ are invariant under the differential $d\varphi$:
    \begin{align*}
        d\varphi(x)E_u(x)=E_u(\varphi(x)),~  d\varphi(x)E_s(x)=E_s(\varphi(x)) \text{ for all } x\in \mathbb{T}^2.
    \end{align*}
    \item The iterates of $\varphi$ in the past are contracting in $E_u$ while in the future in $E_s$. In other words, there are constants $C>0$ and $0<\kappa<\lambda <1$ such that for all $u\in E_u(x), ~v\in E_s(x),  ~x\in \mathbb{T}^2,$ and $n\in \N:$
    \begin{subequations}
    \begin{align}
        \frac{1}{C}\kappa^n \norm{u}\leq \norm{d\varphi^{-n}(x)u}&\leq C\lambda^n\norm{u} \text{ and } \label{negiterate}\\
        \frac{1}{C}\kappa^n \norm{v}\leq \norm{d\varphi^{n}(x)v}&\leq C\lambda^n\norm{v}.\label{positerate}
    \end{align}
    \end{subequations}
\end{itemize}
\end{defn}

\begin{remark}
Note that the contraction properties in equations \eqref{negiterate} and \eqref{positerate} are independent of the metric but $C$ depends on it.
\end{remark}

\begin{defn}\label{bunching}
For $\alpha\in (0,2],$ we call a point $x\in \mathbb{T}^2$ \textit{$\alpha$-bunched} if there exist bunching constants $0<\kappa<\lambda<1$ that provide the bounds of contraction rate in equations \eqref{negiterate} and \eqref{positerate} and $\lambda^{2/\alpha}\leq \kappa$. A diffeomorphism $\varphi$ is \textit{$\alpha$-bunched} if $\sup_{x\in \mathbb{T}^2} \lambda^2\kappa^{-\alpha}<1.$
\end{defn}

\begin{defn}[\cite{H89}]
A function $f:\R\to \R$ is called \textit{$\beta$-H\"older} at a point $x_0$ if $f$ is $\floor{\beta}$ times differentiable and its $\floor{\beta}^{th}$ derivative is H\"older continuous at $x$ with H\"older exponent $\beta-\floor{\beta}$. In particular, for some $\delta>0$ and all $\verts{x_0-x}\leq \delta$ there exits $K>0$ such that 
\begin{align*}
    \verts{f^{(\floor{\beta})}(x_0)-f^{(\floor{\beta})}(x)}\leq K \verts{x_0-x}^{\beta-\floor{\beta}}.
\end{align*}
\end{defn}
\begin{defn}
A $d$-dimensional \textit{distribution} $D$ on a manifold $M$ is a rank $d$ subbundle of the tangent bundle $TM.$
\end{defn}

It means that at every point $x\in M$ there exists a neighborhood $U\subset M$ of $x$ such that $D$ is spanned by vector fields $v_1,\dots, v_d$ that are linearly independent at every point of $U.$
\begin{defn}
A distribution is said to be \textit{$\beta$-H\"older} if it is generated by vector fields whose coeffiecients in local coordinates are $\beta$-H\"older.
\end{defn}

\begin{remark}
By regularity of the distribution $E_u$, we mean regularity of the map $x\mapsto E_u(x).$
\end{remark}

\subsection{H\"older continuity}\label{holderc}
For simplicity, consider $\varphi$ to be a volume preserving Anosov diffeomorphism
\begin{align*}
    \varphi: \mathbb{T}^2\to \mathbb{T}^2
\end{align*}
which satisfies
\begin{enumerate}[label=(\arabic*),noitemsep]
    \item  \label{fixpoint}$\varphi(0)=0$
    \item \label{22}the differential of the map at the origin is: 
    \begin{align}
        d\varphi(0)=\begin{pmatrix}2&0\\0&2^{-1}\end{pmatrix}.\label{differential} 
    \end{align}
\end{enumerate}

\begin{remark}
Note that $\varphi$ satisfies the Definition \ref{anosovdiff} and \ref{bunching}. However, we won't explicitly use the properties in the definitions before \S\ref{general}.
\end{remark}

In this subsection, we will prove the following statement:
\begin{prop}[\cite{H89}]\label{prop2.8}
Suppose $\alpha\in(0,1]$ and $\varphi$ is an Anosov diffeomorphism that satisfies \ref{fixpoint} and \ref{22}. Then, for any $\epsilon>0$, the unstable distribution is $(\alpha-\epsilon)$-H\"older continuous at the origin.
\end{prop}

Since one dimensional distribution $v$ of $\mathbb{T}^2$ consists of lines at each point $x\in \mathbb{T}^2$, we can associate to the distribution a \textit{slope function} $\theta_{v}: \mathbb{T}^2\to \R$ that gives slope of the lines.  Note that we can choose a neighborhood of the origin where the unstable and stable manifolds are uniformly transverse. Therefore, the slope function is well-defined which would be false for vertical lines. Also, define
\begin{align*}
    \norm{v (x)}\coloneqq \verts{\theta_v(x)}.
\end{align*}

\begin{defn}
We say that two distributions $a$ and $b$ on $\mathbb{T}^2$ are \textit{$\delta$-close} to each other for $\delta>0$ if for all $x\in \mathbb{T}^2$ the slope functions $\theta_a$ and $\theta_b$ associated to $a$ and $b$ satisfy
\begin{align*}
    \verts{\theta_a(x)-\theta_b(x)}\leq \delta.
\end{align*}

\end{defn}

For $\delta>0$, define
\begin{align*}
    V(\delta)\coloneqq \{\text{one dimensional distributions on } \mathbb{T}^2 \text{ that are } \delta \text{-close to } E_u\}.
\end{align*}

For the $n^{th}$ iterate of $\varphi$, define an action $\mathcal{T}_n$ on $V(\delta)$ for all $x\in \mathbb{T}^2$ and $v\in V(\delta)$ as
\begin{align}
    (\mathcal{T}_nv)(x)\coloneqq d\varphi^n(p)(v(p))\label{transform}
\end{align}
where $p=\varphi^{-n}(x).$

\begin{remark}\label{rm2.10}
Note that $\mathcal{T}_n (V(\delta))\subset V(\delta)$ and $\mathcal{T}_nv$ approaches $E_u$ for $v\in V(\delta)$ as $n$ tends to infinity. In fact, let $p=\varphi^{-n}(x).$ Then the component of $v(p)$ in the stable direction $E_{u}(p)$ vanishes under the action of $d\varphi^n(p)$ by our assumption in equation \eqref{positerate}. Meanwhile, the component in the direction of $E_u(p)$ stays in unstable direction since
\begin{align*}
    d\varphi^n(p)E_u(p)=E_u(\varphi^n(p))=E_u(\varphi^n (\varphi^{-n}(x)))=E_{u}(x).
\end{align*}
\end{remark}

Corresponding to the Anosov diffeomorphism $\varphi$, let us call $W_s(x)$ and $W_u(x)$ to be the `local' stable and unstable manifolds at $x\in \mathbb{T}^2$. In fact, let $B_\sigma(x)$ be the set of points in $\mathbb{T}^2$ that are at distance $\sigma>0$ from $x\in \mathbb{T}^2$. Fix small $\sigma>0$, and consider the intersection of $W_u(x)$ and $W_s(x)$ with $B_\sigma(x)$ such that $W_u(x)\cap ~W_s(x)=\{x\}$ in the ball.  

For simplicity, assume that $W_u(0)$ and $W_s(0)$ are horizontal and vertical respectively since we can reduce a general case to this setting, see Remark \ref{rmkholder}. By an abuse of notation write $(0,y)\in W_s$ as $y$. Then we have the following statement:

\begin{lem}\label{proofofhc}
For any $\epsilon>0$, there exist $K>0$, $0<\epsilon_1<\epsilon,$ $\delta>0$, and $N\in \N$ such that if $v\in V(\delta)$, $y\in W_s(0)$ with $\verts{y}<\epsilon_1$ then for all $n\geq N$ with $z\coloneqq \varphi^n(y)$,
\begin{align}
    \norm{v(0,y)}<K\verts{y}^{\alpha-\epsilon}\implies \norm{(\mathcal{T}_nv)(z)}<K\verts{z}^{\alpha-\epsilon}.
\end{align}
\end{lem}

\begin{proof}
Recall that the differential of the map $\varphi^n$ at $0$ is:
\begin{align}
    d\varphi^n(0)=\begin{pmatrix}2^n&0\\0&2^{-n}\end{pmatrix}.\label{2.9}
\end{align}
Note that for a point $(0,y)\in W_s(0)$ there is no perturbation in the first component with respect to the origin. Therefore, the second entry of the first row of $d\varphi^n(y)$ has to be $0.$ But we could still have a non-zero component $\gamma_n (y)$ depending on $y$ in the first entry of the second row. Therefore,
\begin{align}
    d\varphi^n(y)=\begin{pmatrix}2^n+\mathcal{O}(\verts{y})&0\\\gamma_n (y)&2^{-n}+\mathcal{O}(\verts{y})\end{pmatrix}.\label{2.10}
\end{align}
Observe that $\gamma_n(0)=0.$ We claim that $\gamma_n(y)\leq P(n)\verts{y}$ for all $n$ where $P(n)$ is a possibly increasing function of $n.$

Meanwhile,
\begin{subequations}
\begin{align}
    (\mathcal{T}_nv)(z)=d\varphi^n(y)v(y)&=\begin{pmatrix}2^n+\mathcal{O}(\verts{y})&0\\\gamma_n (y)&2^{-n}+\mathcal{O}(\verts{y})\end{pmatrix}\begin{pmatrix}1\\ \theta_v(y)\end{pmatrix}\\
    &=\begin{pmatrix}2^n+\mathcal{O}(\verts{y})\\ \gamma_n(y)+(2^{-n}+\mathcal{O}(\verts{y}))\theta_v(y)\end{pmatrix}.\label{coneaction}
\end{align}
\end{subequations}
Because we are interested in the slopes, the right hand side is equivalent to \\$\begin{pmatrix}1\\ (\frac{1}{2^n+\mathcal{O}(\verts{y})}\gamma_n(y)+\frac{2^{-n}+\mathcal{O}(\verts{y})}{2^n+\mathcal{O}(\verts{y})}\theta_v(y)\end{pmatrix}.$ Therefore, it suffices to get the bound on the second entry. In fact, if we pick small $\verts{y}$,
\begin{subequations}
\begin{align}
    \norm{(\mathcal{T}_nv)(z)}&=\verts*{\frac{1}{2^n+\mathcal{O}(\verts{y})}\gamma_n(y)+\frac{2^{-n}+\mathcal{O}(\verts{y})}{2^n+\mathcal{O}(\verts{y})}\theta_v(y)}\label{slope}\\
    &\leq 2^{-n}(1+2^{-n}\mathcal{O}(\verts{y}))P(n)\verts{y}+(2^{-2n}+2^{-n}\mathcal{O}(\verts{y}))K\verts{y}^{\alpha-\epsilon}.\label{actionbound}
\end{align}
\end{subequations}
By the mean value theorem, we have $\verts{y}\leq L 2^n\verts{\varphi^n (y)}.$ Therefore,
\begin{align}
    \norm{(\mathcal{T}_nv)(z)}\leq& P(n)L\verts{z}+2^{-n}P(n)\mathcal{O}(\verts{y})\verts{z}
    \nonumber\\&+(2^{-2+\alpha})^n2^{-\epsilon n}KL^{\alpha-\epsilon}\verts{z}^{\alpha-\epsilon}+(2^{\alpha-1})^n L^{\alpha-\epsilon}2^{-\epsilon n}\mathcal{O}(\verts{y})K\verts{z}^{\alpha-\epsilon}.\label{boundtransform}
\end{align}
Since $\verts{y}\leq \epsilon_1$, $\mathcal{O}(\verts{y})<c$ for some constant $c>1$. Now take $N$ large enough such that $c2^{-\epsilon N}L^{\alpha-\epsilon}\leq \frac{1}{4}.$ Note that for $\alpha\in(0,1]$, $2^{-2+\alpha}\leq 1$ and $2^{\alpha-1}\leq 1.$ Therefore, the last two terms contribute at most $\frac{K}{2}\verts{z}^{\alpha-\epsilon}$. Further, for $n\in [N,2N]$, choose $K> 8 c P(n)L$. Combining the preceding with the facts that $\varphi$ contracts $y$ when $\verts{y}\leq \epsilon_1<1$ and $\verts{z}\leq \verts{z}^{\alpha-\epsilon}$ for $\alpha\in (0,1]$, we get
\begin{align}
    \norm{(\mathcal{T}_nv)(z)}\leq K\verts{z}^{\alpha-\epsilon}\label{holderbound}
\end{align}
for all $n\in[N,2N].$ Now we can use induction to prove the inequality \eqref{holderbound} for all $n\geq N.$

What remains to prove is the claim that $\gamma_n(y)\leq P(n)\verts{y}$ for all $n.$ In fact, using the chain rule $d\varphi^{n+m}(y)=d\varphi^n(\varphi^m(y))d\varphi^m(y).$ Therefore,
\begin{align*}
    &\begin{pmatrix}2^{n+m}+\mathcal{O}(\verts{y})&0\\\gamma_{n+m}(y)&2^{-n-m}+\mathcal{O}(\verts{y})\end{pmatrix}\\=&\begin{pmatrix}2^{n}+\mathcal{O}(\verts{y})&0\\\gamma_{n}(\varphi^m(y))&2^{-n}+\mathcal{O}(\verts{y})\end{pmatrix}\begin{pmatrix}2^{m}+\mathcal{O}(\verts{y})&0\\\gamma_{m}(y)&2^{-m}+\mathcal{O}(\verts{y})\end{pmatrix}
\end{align*}
which implies
\begin{align*}
    \gamma_{n}(y)=\gamma_1(\varphi^{n-1}(y))(2^{n-1}+\mathcal{O}(\verts{y}))+\parens*{\frac{1}{2}+\mathcal{O}(\verts{y})}\gamma_{n-1}(y).
\end{align*}
Since $\gamma_0(y)=0$ and $\varphi$ is contracting in the stable direction, the claim follows by induction.
\end{proof}

\begin{remark}
The argument used to bound \eqref{boundtransform} already imposes restrictions on $\alpha$ which gives a hint that higher regularity of the foliation is harder to achieve.
\end{remark}

Note that as $n$ goes to infinity $z$ approaches the origin. Therefore, Lemma \ref{proofofhc} does not prove H\"older continuity. Nevertheless, using Lemma \ref{proofofhc} we will prove that $\mathcal{T}_n$ preserves the collection of $(\alpha-~\epsilon)$-H\"older distributions in the stable direction. Using Remark \ref{rm2.10}, we know that $\mathcal{T}_nv$ for $v\in V(\delta)$ converges to $E_u$ as $n\to \infty$. Therefore, by equicontinuity, the limiting distribution, $E_u,$ has to be H\"older in the stable direction at the origin with exponent $\alpha-\epsilon$, see Corollary \ref{cor2.14}. Meanwhile, $E_u$ is as regular as the map $\varphi$ in the unstable direction because the unstable manifold is as regular as the map. Remember that $E_u$ near the origin is the tangent space of the unstable manifold. Since the distribution is regular in both stable and unstable directions, it is regular in the neighborhood of the origin.

For a fixed $\epsilon_1>0$ define
\begin{align}
    V(\delta,\epsilon_0, K) \coloneqq \{v\in V(\delta)\mid \text{for all } y \text{ with }  \epsilon_0\leq \verts{y}\leq\epsilon_1, \verts{\theta_v(y)}\leq K\verts{y}^{\alpha-\epsilon}\}.\label{econefield}
\end{align}

\begin{prop}\label{limitofv}
For any $0<\delta<1$ and $\epsilon_1>0$ there exist positive constants $K>0$, $N\in \Z_+$ and $\epsilon_0<\epsilon_1$ so that, for all $n\in \Z_+$,
\begin{align}
    \mathcal{T}_{nN} (V(\delta))\subset V(\delta, \epsilon_0 2^{-n},K).\label{2.11}
\end{align}
\end{prop}
\begin{proof}
For fixed $\epsilon_0<1$ and $\delta>0$, we claim that $V(\delta)\subset V(\delta,\epsilon_0,K).$ In fact, the uniform continuity of $v\in V(\delta)$ implies that $\theta_v(y)$ is bounded. Therefore, we can choose large $K$ such that $v$ satisfies the condition in Definition \ref{econefield}.

For fixed $\epsilon_1>0$, choose $0<\epsilon_0<\epsilon_1$ such that $\varphi(0,\epsilon_1)>\epsilon_0$ and $\varphi(0,-\epsilon_1)<-\epsilon_0.$ It means the point $(0,\epsilon_1)$ remains $\epsilon_0$-away from the origin after propagation. Note that $\varphi^n(\{(0,y)|\epsilon_0< ~\verts{y}<\epsilon_1\})$ covers the punctured neighborhood of the origin.

Fix $v\in V(\delta)$ and $n\in \Z_+$. We want to prove that $\mathcal{T}_{nN}v\in V(\delta,\epsilon_02^{-n},K).$ Consider $y$ such that $\epsilon_02^{-n}<\verts{y}<\epsilon_1.$ From our choice of $\epsilon_0$ we know that there exists $m\leq n$ such that for some $y'$ with $\epsilon_0<\verts{y}<\epsilon_1$ and $\varphi^n(0,y')=(0,y).$ Because $\varphi^n(0,y')$ approaches to $0$, we can apply Lemma \ref{proofofhc} iteratively to $\varphi^{iN}(0,y')$ for $0\leq i< m.$ Note that $\theta_v(y)\leq K\verts{y}^{\alpha-\epsilon}$ implies that the condition holds for $z=\varphi^m(0,y')=(0,y)$ by induction. 
\end{proof}

\begin{cor}\label{cor2.14}
The unstable distribution is $(\alpha-\epsilon)$-H\"older continuous at the origin.
\end{cor}
\begin{proof}
Choose $N$, $K$ and $\delta>0$ and $\epsilon_1>0$ as in Proposition \ref{limitofv}. Then for $\verts{y}<\epsilon_1$ we have
\begin{align*}
    \verts{\theta_{E_u}(y)}=\lim_{n\to\infty}\verts{\theta_{\mathcal{T}_{nN}v}(y)}\leq K \verts{y}^{\alpha-\epsilon}
\end{align*}
because the inequality holds for arbitrary $0<\epsilon_02^{-n}<\verts{y}< \epsilon_1$ as we can choose $nN$ to be very large.
\end{proof}

\begin{remark}\label{rmkholder}
In a general setting, the stable and unstable manifolds are not the coordinate axes. Nevertheless, they are transverse. Since the manifolds are as smooth as the map $\varphi$, we can find a smooth coordinate map $\xi:B_\sigma(0)\to[-\epsilon,\epsilon]\times[-\epsilon,\epsilon]$ such that points in $W_s'\coloneqq \xi(W_s)$ have the first component $0$ while those in $W_u'\coloneqq \xi(W_u)$ have the second component $0.$ In other words, we can straighten out the stable and unstable manifolds. 

Now let $\phi$ be the push forward of the Anosov diffeomorphism $\varphi$ with respect to the coordinate $\xi.$ From the preceding discussion, we know that the unstable distribution associated to $\phi$ is H\"older. Note that the unstable distribution of $\varphi$ is the pull back of the one for $\phi$. The subtlety here is $d\phi(0)$ changes to $\begin{pmatrix}\lambda_1&0\\0&\lambda_2\end{pmatrix}$ for some positive $\lambda_i.$ However, the volume preserving property of $\varphi$ implies that $\lambda_1=\lambda_2^{-1}.$ Therefore, the argument we gave passes through. The only change is $2$ gets replaced by $\lambda_1>1$. This proves Proposition \ref{prop2.8}.
\end{remark}

\subsection{Differentiability}\label{diff}
Assuming $\varphi$ satisfies the properties in Proposition \ref{prop2.8}, we have the following statement:

\begin{prop}\label{propdiff}
The unstable distribution on $\mathbb{T}^2$ associated to $\varphi$ is differentiable at the origin.
\end{prop}
A one variable function $f:\R\to \R$ is differentiable means that $\frac{f(x+h)-f(x)}{h}$ approaches to a limit, say $f'(x)$, as $h$ goes to zero. Note that the sign of $h$ does not matter. Thus, for $h_1,h_2>0$,
\begin{align*}
    \lim_{h_1\to 0}\frac{f(x+h_1)-f(x)}{h_1}&=f'(x)\\
    \lim_{h_2\to 0}\frac{f(x)-f(x-h_2)}{h_2}&=f'(x)
\end{align*}
which is same as the statement
\begin{align}
    \lim_{h_1,h_2\to 0}\frac{1}{h_1h_2}\verts*{h_2f(x+h_1)+h_1f(x-h_2)-(h_1+h_2)f(x)}=0.\label{differentiation}
\end{align}
It is clear that a function satisfying equation \eqref{differentiation} is differentiable since the difference quotient satisfies a Cauchy criterion.

Similar to the previous section, proving differentiability is tantamount to proving differentiability of $\theta_{E_u}$ along the stable direction. Note that Proposition \ref{prop2.8} already implies Lipschitz continuity.

Again without loss of generality, assume that $W_u(0)$ and $W_s(0)$ are straightened out. By an abuse of notation, we write $(0,y)\in W_s(0)$ as $y$.
\begin{lem}\label{lemdiff}
For all $\epsilon'>0$, there exit $\epsilon>0$ and $K>0$ such that if 
\begin{align}
    \verts{h_1\theta_{E_u}(h_2)+h_2\theta_{E_u}(-h_1)}\leq Kh_1h_2(h_1+h_2)^{1-\epsilon'}\label{lem2.14}
\end{align}
where $0<h_i<\epsilon$ then
\begin{align}
    \verts{\tilde{h}_1\theta_{E_u}(\tilde{h}_2)+\tilde{h}_2\theta_{E_u}(-\tilde{h}_1)}\leq K\tilde{h}_1\tilde{h}_2(\tilde{h}_1+\tilde{h}_2)^{1-\epsilon'}\label{lem2.15}
\end{align}
where $(-1)^i\tilde{h}_i=\varphi((-1)^ih_i)$.
\end{lem}
\begin{remark}
Note that Lemma \ref{lemdiff} implies Proposition \ref{diff} when $W_u(0)$ and $W_s(0)$ are straightened out. In particular, we can use the argument given in Proposition \ref{limitofv} with a modification:
\begin{align}
    V(\delta,\epsilon_0, K) \coloneqq \{&v\in V(\delta)\mid \text{ for all } h_i \text{ with }  \epsilon_0\leq \verts{h_i}\leq\epsilon_1,\theta_v \text{ satisfies } \eqref{lem2.14}\}.
\end{align}
A general case follows from the straightening argument given in Remark \ref{rmkholder}.
\end{remark}
\begin{proof}[Proof of Lemma \ref{lemdiff}]
Using the invariance of $E_u$ under the action of $\varphi$, we know that $E_u(\tilde{h})=d\varphi(h) E_u(h)$ where $\tilde{h}=\varphi(h)$. Because we are interested in the slopes, the discussion after equation \eqref{coneaction} implies
\begin{align}
    \theta_{E_u}(\tilde{h})&=\frac{\gamma_1(h)}{2+\mathcal{O}(h)}+\frac{2^{-1}+\mathcal{O}(h)}{2+\mathcal{O}(h)}\theta_{E_u}(h)\nonumber\\
    &=2^{-2}\theta_{E_u}(h)+ch+\mathcal{O}(\verts{h}^2)\label{2.18}
\end{align}
where $c$ is a constant. Moreover, $ch$ in the second line takes into account the second order expansion $\gamma_1(h)=a_1h+\mathcal{O}(h^2)$ at the origin since $\gamma_1(0)=0.$

Note that $\tilde{h}_i=\frac{1}{2}h_i+\mathcal{O}(h_i^2)$. Therefore,
\begin{align}
    \verts{\tilde{h}_1\theta_{E_u}(\tilde{h}_2)+\tilde{h}_2\theta_{E_u}(-\tilde{h}_1)}\leq& \frac{1}{4}\verts*{\tilde{h}_1\theta_{E_u}(h_2)+\tilde{h}_2\theta_{E_u}(-h_1)}+\mathcal{O}(\tilde{h}_1\tilde{h}_2(\tilde{h}_1+\tilde{h}_2))\nonumber\\
    =& \frac{1}{8}\verts*{h_1\theta_{E_u}(h_2)+h_2\theta_{E_u}(-h_1)} +\mathcal{O}(\tilde{h}_1\tilde{h}_2(\tilde{h}_1+\tilde{h}_2))\nonumber\\
    \leq& \frac{1}{8}Kh_1h_2(h_1+h_2)^{1-\epsilon'}+\mathcal{O}(\tilde{h}_1\tilde{h}_2(\tilde{h}_1+\tilde{h}_2))\nonumber\\
    \leq &2^{-\epsilon'}K\tilde{h}_1\tilde{h}_2(\tilde{h}_1+\tilde{h}_2)^{1-\epsilon'}+\mathcal{O}(\tilde{h}_1\tilde{h}_2(\tilde{h}_1+\tilde{h}_2)^{1-\epsilon'}).
\end{align}
Now choose $K$ large enough such that the second term is bounded by $(1-2^{-\epsilon'})K\tilde{h}_1\tilde{h}_2(\tilde{h}_1+\tilde{h}_2)^{1-\epsilon'}$.
\end{proof}

\subsection{H\"older continuity of the derivative} \label{holdercd}
Assuming $\varphi$ satisfies the properties in Proposition \ref{prop2.8}, we have the following statement:
\begin{prop}
For $\alpha\in (1,2]$ and $\epsilon>0$, the unstable distribution on $\mathbb{T}^2$ associated to $\varphi$ is $C^{\alpha-\epsilon}$ at the origin.
\end{prop}
We already proved that the distribution is differentiable. Now we need to prove that the first derivative is $\beta\coloneqq (\alpha-1-\epsilon)$-H\"older continuous at the origin.
\begin{proof}
In the spirit of Lemma \ref{proofofhc} and Proposition \ref{limitofv}, it is enough to prove that for any small $y\in W_s(0)$ if $\verts{\partial_y\theta_{E_u}(y)-\partial_y\theta_{E_u}(0)}\leq K\verts{y}^{\beta}$ then
\begin{align}
    \verts{\partial_z\theta_{E_u}(z)-\partial_z\theta_{E_u}(0)}\leq K\verts{z}^{\beta}\label{prop2.18.1}
\end{align}
where $z=\varphi^{n}(y).$

Assume that $W_s(0)$ and $W_u(0)$ are straightened out as we can do that for a general case. The differentiability of $\theta_{E_u}$ implies that, near origin,
\begin{align}
    \theta_{E_u}(y)=cy+h(y)y\label{taylorexpansion}
\end{align}
for some constant $c$ and a function $h(y)$ that vanishes as $y\to 0$ whence $\partial_y{}^{N}\theta_{E_u}(0)=0$ where ${}^{N}\theta_{E_u}(y)=h(y)y$ is the nonlinear term of $\theta_{E_u}(y)$. Similarly, we can expand $\gamma_n(y)$ as
\begin{align}
    \gamma_n(y)=b_ny+l_n(y)y\label{taylorexpansion2}
\end{align}
where $b_n$ is a constant and $l_n(y)\to 0$ as $y\to 0.$

It is clear that the linear term of $\theta_{E_u}$ satisfies the inequality \eqref{prop2.18.1} for small $y$. Therefore, we just need a bound for the non-linear term. In other words, $\verts{\partial_y{}^N\theta_{E_u}(y)}\leq K\verts{y}^\beta$ implies $\verts{\partial_z{}^N\theta_{E_u}(z)}\leq K\verts{z}^\beta$. 

In fact, the discussion after equation \eqref{coneaction} implies
\begin{align*}
    \theta_{E_u}(\varphi^n y)=(2^{-n}+2^{-2n}\mathcal{O}(\verts{y}))\gamma_n(y)+(2^{-2n}+2^{-n}\mathcal{O}(\verts{y}))\theta_{E_u}(y)
\end{align*}
whence the non-linear term of $\theta_{E_u}(z)$ is
\begin{align}
    ^N\theta_{E_u}(z)=&2^{-n}l_n(y)y+2^{-2n}\mathcal{O}(\verts{y})\gamma_n(y)+2^{-n}\mathcal{O}(\verts{y})cy\nonumber\\
    &+2^{-2n}h(y)y+2^{-n}\mathcal{O}(\verts{y})(h(y)y).\label{thetanonlinear}
\end{align}

Meanwhile, the chain rule implies $\partial_y(^N\theta_{E_u})(z)=\partial_z(^N\theta_{E_u})(z)\partial_y\varphi^{n}(y)$. Therefore,
\begin{align}
    \partial_z(^N\theta_{E_u})(z)=\frac{\partial_y(^N\theta_{E_u})(z)}{\partial_y\varphi^{n}(y)}.\label{nonlinear}
\end{align}

Remember that the mean value theorem implies $\verts{y}\leq L2^n\verts{\varphi^n(y)}.$ Therefore, by assumption,
\begin{align}
    2^{-n}\verts{\partial_y(h(y)y)}&\leq 2^{-n}K\verts{y}^{\beta}\nonumber\\
    &\leq K L^{\beta}2^{n(\beta-1)}\verts{z}^{\beta}\nonumber\\
    &=K L^{\beta}2^{n(\alpha-2-\epsilon)}\verts{z}^{\beta}.\label{2.24}
\end{align}
As $\alpha\in (1,2],$ the power of $2$ is negative. Therefore, we can take $N\in \N$ large enough such that $L^\beta 2^{n(\alpha-2-\epsilon)}\leq \frac{1}{4}$ for all $n\geq N$ which gives a bound in \eqref{nonlinear} corresponding to the last two terms in \eqref{thetanonlinear}. On the other hand, rest of the terms in \eqref{thetanonlinear} contribute only $\mathcal{O}(\verts{y})$ term. Therefore, picking $K$ large enough as in Lemma \ref{proofofhc} so that $2\mathcal{O}(\verts{y})\leq K\verts{z}^\beta$, we get a bound for $n\in [N,2N]$ which implies the proposition after passing through an inductive step.
\end{proof}


\section{Generalization}\label{general}
Now that we have proven regularity in two dimensions for volume preserving diffeomorphisms that fix the origin and whose differential has diagonal entries $2$ and $2^{-1}$, we are ready to comment on the generalization of \S\ref{sec2}. We leave the detail of the proof to the reader.

\subsection{Regularity at a general point}

For any $x_0\in \mathbb{T}^2$, we can find a smooth coordinate map $\chi_{\varphi^{n}(x_0)}:U_{\varphi^{n}(x_0)}~\to \R^2$ in a neighborhood $U_{\varphi^{n}(x_0)}\subset \mathbb{T}^2$ of $\varphi^n (x_0)$ such that $\chi_{\varphi^{n}(x_0)}(\varphi^{n}(x_0))=0$ and the differential of $\chi_{\varphi^{n}(x_0)}$ isometrically sends the tangent spaces $E_u(\varphi^{n}(x_0))$ and $E_s(\varphi^{n}(x_0))$ to the horizontal and vertical coordinate axes of $\R^2$ respectively. Now consider a family of composite maps
\begin{align}
    \psi_{x_0,n}\coloneqq\chi_{\varphi^{n+1}(x_0)}\circ \varphi\circ \chi_{\varphi^{n}(x_0)}^{-1}.
\end{align}
Note that $\psi_{x_0,n}(0)=0.$ Further, $\psi_{x_0,n}$ is volume preserving and $d\psi_{x_0,n}(0)$ has diagonal entries $\eta$ and $\eta^{-1}$. After straightening out the stable and unstable manifolds associated to $\psi_{x_0,n}$ if necessary, \S\ref{sec2} implies that the unstable (stable) distribution associated to $\psi_{x_0,n}$ is regular. Note that an unstable manifold associated to $\varphi$ is the pullback of an unstable manifold of $\psi_{x_0,0}$ using $\chi_{x_0}.$ Now the smoothness of the coordinate maps implies regularity of the distribution associated to $\varphi.$

\subsection{Regularity when diagonal entries vary}
In section \S\ref{sec2}, we never used the power of $\alpha$-bunching, \eqref{negiterate} and \eqref{positerate} since the diagonal entries of $d\varphi^n(0)$ were fixed which implies bunching. However, we can work with diagonal entries of $d\varphi^n(0)$: $\eta_n<1$ and $\eta_n^{-1}$ that vary with $n$ but are bounded to satisfy the bunching condition. In other words, for each $x\in \mathbb{T}^2$, there exist positive constants $0<\kappa<\lambda<1$ such that $\kappa^n\leq \eta_n\leq \lambda^n$ and $\sup_{x\in M}\lambda^2\kappa^{-\alpha}<1.$

Some minor modifications in the proof are:
\begin{itemize}
    \item \textbf{H\"older continuity:} We replace $2^{-n}$ with $\lambda^{n}$ in the bound \eqref{actionbound}. Further, using the mean value theorem, $\verts{y}\leq L\eta_n^{-1}\verts{\varphi^n(y)}\leq L\kappa^{-n}\verts{\varphi^n(y)}.$ Therefore, the inequality \eqref{boundtransform} becomes
    \begin{align}
        \norm{(\mathcal{T}_nv)(z)}\leq P(n)L\lambda^n\kappa^{-n}\verts{z}+(\lambda^{2}\kappa^{-\alpha})^n\kappa^{\epsilon n}KL^{\alpha-\epsilon}\verts{z}^{\alpha-\epsilon}.\label{boundtrans2}
    \end{align}
    Choose $N\in \N$ large such that $L^{\alpha-\epsilon}\kappa^{\epsilon n}\leq \frac{1}{2}$ for $n\geq N$. And pick $K$ larger than $LP(n)\lambda^n\kappa^{-n}$ for $n\in [N,2N].$ With bunching condition, we get the H\"older continuity as the estimate for $\gamma_n(y)$ is still $P(n)\verts{y}$.
    
    \item \textbf{Differentiability:} In the proof of Lemma \ref{lemdiff}, we replace $2^{-1}$ with $\eta_1$.
    
    \item \textbf{H\"older continuity of the derivative:} The changes are same to that for H\"older continuity. In fact, $(\lambda^2\kappa^{-\alpha})^n$ will appear in the inequality \eqref{2.24} where we have to use bunching. Finally, we have to take $\lambda^n\kappa^{-n}$ into account while choosing a large $K$.
\end{itemize}

\subsection{Regularity in higher dimensions}\label{sec3.3}
The idea of the proof in higher dimensions is same with justification of some ingredients that we took for granted in two dimensions.

Like we have been doing, we can straighten out the unstable (stable) manifolds and work with Anosov diffeomorphisms $\varphi$ that fix the origin. In this setting, for any $y$ in stable manifold close to origin, the differential of $\varphi^n$ is
\begin{align*}
    d\varphi^n(y)=\begin{pmatrix}A_n&0\\\Gamma_n&B_n\end{pmatrix}
\end{align*}
where $A_n$ and $B_n$ are matrices corresponding to the unstable and stable directions such that $\Gamma_n(0)=0$, $\norm{A_n^{-1}(0)}\leq L\lambda^n$ and $\norm{B_n(0)}\leq L\lambda^n.$

\begin{itemize}
    \item \textbf{H\"older continuity:} In contrast to lines in $\mathbb{T}^2$, the distributions have higher dimensions. However, we can think of an unstable (stable) distribution at each point as a linear map $T: \R^{d_u}\to \R^{d_s}$ where $d_u$ and $d_s$ are the dimensions of the stable and unstable manifolds. Instead of slope function, we can use $\begin{pmatrix}\mathds{1}\\T\end{pmatrix}$ to characterize $T$ where $\mathds{1}$ is the identity. Then, for small $y\in W_s(0)$, \eqref{slope} becomes
    \begin{align*}
        \norm{(\mathcal{T}_nv)(z)}\leq \norm{\Gamma_n(y)}\norm{A_n^{-t}(y)}+\norm{B_n(y)}\norm{T(y)}\norm{A_n^{-1}(y)}.
    \end{align*}
    The proof of the bound $\norm{\Gamma_n(y)}\leq P(n)\norm{y}$ is similar to that for $\gamma_n$. However, we can't take 
    \begin{subequations}
    \begin{align}
        \label{3.2}
        \norm{A_n^{-1}(y)}&\leq L \lambda^n\\
        \label{3.1}
        \norm{B_n(y)}&\leq L\lambda^n\\
        \label{3.3}
        \norm{y}&\leq L\kappa^{-n}\norm{\varphi^n(y)}
    \end{align}
    \end{subequations}
    for granted because norms don't have to work like numbers.
    
    Note that \eqref{3.2} and \eqref{3.1} are similar. The third one follows from the mean value theorem and a similar bound for $\norm{B_n^{-1}(y)}.$ Therefore, it suffices to prove the second one. Meanwhile, it is enough to show that there exists $R>1$ such that
    \begin{align}
        \norm{B_n(y)-B_n(0)}\leq R\lambda^n\norm{y}.\label{boundb}
    \end{align} 
    
    \begin{proof}
    We proceed by induction: We can find $R'>1$ such that
    \begin{align*}
        \norm{B_1(y)-B_1(0)}\leq R'\norm{y}
    \end{align*}
    for a constant $R'$ as $B_1$ is regular and $\norm{B_1(0)}\leq R'\lambda^n$. Because $\varphi$ is contracting, we know that $\norm{\varphi^n y}\leq R' \omega^n\norm{y}$ for some $\omega<1.$ Assume that the statement is true for all $i\leq n$.
    
    Using $d\varphi^{n+m}(y)=d\varphi^n(\varphi^m(y))d\varphi^m(y),$ it is clear that
    \begin{align*}
        B_{n+1}(y)-B_{n+1}(0)=B_{n}(\varphi(y))B_1(y)-B_n(0)B_1(0)
    \end{align*}
    To use the induction argument, first note that the right hand side can be written as a telescoping sum:
   \begin{align*}
       &B_n(\varphi(y))\parens*{B_1(y)-B_1(0)}\\
       &+\sum_{j=1}^n B_{n-j}(\varphi^{j+1}(y))\parens*{B_1(\varphi^{j}(y))-B_1(0)}B_j(0)\\
       &+(B_1(\varphi^n(y))-B_1(\varphi^n(0)))B_n(0).
   \end{align*}
    The first term in this telescoping sum can be bounded using the base case. The bound for third term follows from the base case and $\norm{\varphi^n y}\leq R' \omega^n\norm{y}$. For the second term, note that
    \begin{align*}
        \norm{B_i(\varphi^my)}&\leq \norm{B_i(0)}+\norm{B_i(\varphi^my)-B_i{0}}\\
        &\leq R'\lambda^n+R'\lambda^n\norm{\varphi^m(y)}\\
        &\leq R'\lambda^n(1+R\omega^m\verts{y}).
    \end{align*}
    Therefore, using the induction argument and the preceding, it is easy to see that the claim follows  for $R>\frac{2R'^4}{\lambda(1-\omega)}$.
    \end{proof}
    Now that we have justified the bound in \eqref{3.2}, \eqref{3.1} and \eqref{3.3}, the proof of H\"older continuity follows easily as in two dimensions with varying diagonal entries.
    
    \item \textbf{Differentiablity:} For a function $f:\R^n\to 
\R$ to be differentiable, instead of \eqref{differentiation}, we use the criterion
    \begin{align*}
        \frac{1}{h_1h_2h_2}\verts{h_2h_3f(c_1(h_1))+h_1h_3f(c_2(h_2))+h_1h_2f(c_3(h_3))\\-(h_1h_2+h_1h_3+h_1h_2)f(x)}\to 0
    \end{align*}
    as $h_1,h_2,h_3\to 0$ for geodesics $c_i$ on $\R^n$ starting at $x$ with initial direction $v_i$ such that $\sum v_i=0.$
    
    Just like in the proof of H\"older continuity, we use $T$ in the form $\begin{pmatrix}\mathds{1}\\T\end{pmatrix}$ instead of the slope function. Remember $E_u(0)$ is one dimensional in two dimensions, so the geodesics $c_i(h_i)$ and $\tilde{c}_i(\tilde{h}_i)$ (to be made precise soon) can be replaced with of $h_i$ and $\tilde{h}_i.$ However, we have to use geodesics in the arguments of $\theta_{E_u}$ in \eqref{2.18} as we could start pointing in any direction.
    
    Because we are interested in the direction of the geodesics, define $\mathcal{E}_{s}$ to the collection of $(v_1,v_2,v_3)$ such that $v_i\in E_s,$ $\sum v_i=0$ and $\sum \verts{v}_i=1.$ For $x\in M$, define the action of $\varphi$ on $\mathcal{E}_s$ as:
    \begin{align}
    (\mathcal{T}_nv_i)(x)\coloneqq \frac{d\varphi^n(p)(v_i(p))}{r_n}
    \end{align}
    where $p=\varphi^{-n}(x)$ and $r_n$ is a normalization factor. Note that \eqref{positerate} implies $\frac{1}{L}\kappa^n\leq r_n \leq L\lambda^n.$ Using $\tilde{h}_i=r_nh_i$ instead of $\varphi^n((-1)^ih_i)$ in Lemma \ref{lemdiff}, we suppose $\tilde{c}_i(\tilde{h}_i)$ starts in the direction $(\mathcal{T}_nv_i)(x).$
    
    However, to use the arguments in Lemma \ref{lemdiff}, we use $\varphi^n(\bar{c}_i(h_i))$ where $\bar{c}_i$ starts in the direction $v_i(\varphi^{-n}(x))$. Since $\tilde{c}_i$ and $\varphi^n(\bar{c}_i)$ are tangent at $0,$ the error in using $\varphi^n(\bar{c}_i)$ instead of $\tilde{c}_i$ is of order $o(\tilde{h}_1\tilde{h}_2\tilde{h}_3).$
    
    \item \textbf{H\"older continuity of the derivative:} In addition to expanding $\Gamma_n$ and $T$ (the `slope' function in higher dimensions) after writing it in the form $\begin{pmatrix}\mathds{1}\\T\end{pmatrix}$, we expand $A_n$ and $B_n$ into linear and non-linear terms. But the proof is similar with extra terms to carry around.
\end{itemize}

\subsection{Regularity for Anosov flows}
In this section, we will briefly comment on how to prove regularity for volume preserving Anosov flows in three dimensions. 

\begin{defn}
A volume preserving Anosov flow on a compact manifold $M$ is a $C^\infty$ flow $\varphi^t: M\to M$ that preserves the volume of $M$ such that $\dot{\varphi}^t\neq 0$ and the tangent bundle $TM$ splits into $E_u\oplus E_s\oplus E_{\varphi^t}$ where $E_{\varphi^t}$ is the span of $\dot{\varphi}^t$ and $E_u$ and $E_s$ are $d_u$ and $d_s$ dimensional subspaces of $TM$. Further, $E_u$ and $E_s$ are invariant under the flow and contracting as in Definition \ref{anosovdiff} with $n\in \N$ replaced with $t>0$.
\end{defn}
The discussion in \S\ref{sec2} already implies regularity of stable/unstable distributions associated to the flow at periodic points when $M$ is three dimensional. In fact, if $\varphi^{t_0}(x)=x$ for some time $t_0$ then define $\phi\coloneqq \psi^{t_0}$ and carry on the argument for $\phi.$

In general, we work with weak unstable manifold $W_{wu}$ corresponding to $E_u\oplus E_{\varphi^t}$. The smoothness of weak manifolds are guaranteed but not that of strong manifolds associated to just $E_u$ and $E_s$ \cite{J72}. To avoid the flow direction, we pass our arguments to $W_u\coloneqq W_{wu}(x)\cap \mathcal{F}$ where $\mathcal{F}(x)$ for $x\in M$ is a submanifold of $M$ transverse to $\dot{\varphi}^t(x).$ $\mathcal{F}(x)$ is often called Poincar\'e section. Now the proof given in \S\ref{sec3.3} works for $W_u$ after replacing the discrete time $n$ with continuous time $t.$ As the flow direction is smooth, we get regularity of $W_{wu}.$ A small modification that we have to make is in the proof of the claim \eqref{boundb}. To prove \eqref{boundb} for all $t>0$, we have to change the base case to $t<1$ instead of just $n=1.$

\section*{Acknowledgement}
The project was supported by Paul E. Gray UROP fund. I am thankful to Semyon Dyatlov for his guidance throughout the project. And I am grateful to my parents.

\bibliography{main}

\begin{bibdiv}
\begin{biblist}

\bib{D}{article}{
      author={Dyatlov, Semyon},
       title={Notes on hyperbolic dynamics},
        date={2018},
     journal={arXiv preprint \arxiv{1805.11660}},
}

\bib{D21}{misc}{
      author={Dyatlov, Semyon},
       title={Personal communication},
        date={2021},
}

\bib{H89}{thesis}{
      author={Hasselblatt, Boris},
       title={{Regularity of the Anosov splitting and a new description of the
  Margulis measure}},
        type={Ph.D. Thesis},
        date={1989},
}

\bib{HK}{article}{
      author={Hurder, Steve},
      author={Katok, Anatole},
       title={{Differentiability, rigidity and Godbillon-Vey classes for Anosov
  flows}},
        date={1990},
     journal={Publications Math{\'e}matiques de l'Institut des Hautes
  {\'E}tudes Scientifiques},
      volume={72},
      number={1},
       pages={5\ndash 61},
}

\bib{KH}{book}{
      author={Katok, Anatole},
      author={Hasselblatt, Boris},
       title={Introduction to the modern theory of dynamical systems},
   publisher={Cambridge University Press},
        date={1997},
      number={54},
}

\bib{J72}{article}{
      author={Plante, Joseph~F},
       title={Anosov flows},
        date={1972},
     journal={American Journal of Mathematics},
      volume={94},
      number={3},
       pages={729\ndash 754},
}

\end{biblist}
\end{bibdiv}

\textit{E-mail:} \href{mailto:rkoirala@mit.edu}{rkoirala@mit.edu}

\scshape{Department of Mathematics, Massachusetts Institute of Technology, MA 02139}

\end{document}